\documentclass[12pt]{article}
\usepackage{amsmath,amsxtra,amssymb,latexsym, amscd,amsthm}

\advance\voffset-1truecm\relax
\advance\hoffset-1truecm\relax

\def\N{{\Bbb N}}

\def\P{{\Bbb P}}

\def\Q{{\Bbb Q}}

\numberwithin{equation}{section}

\newcommand {\Cal}{\mathcal}

\newtheorem{lemma}{Lemma}[section]
\newtheorem{theorem}[lemma]{Theorem}

\newtheorem{definition}[lemma]{Definition}
\newtheorem{remark}[lemma]{Remark}
\def\leq{\leqslant}

\begin{document}
\title{Schmidt's subspace theorem for moving hypersurface targets}
\author{Nguyen Thanh Son, Tran Van Tan, and Nguyen Van Thin}
\date{$\quad$}
\maketitle
\vspace{-0.5cm}
\begin{abstract} It was discovered that there is a formal analogy between Nevanlinna theory and Diophantine approximation.
Via Vojta's dictionary, the Second Main Theorem in Nevanlinna theory corresponds to Schmidt's Subspace Theorem in Diophantine approximation. Recently, Cherry, Dethloff, and Tan (arXiv:1503.08801v2 [math.CV]) obtained a Second Main Theorem for moving hypersurfaces intersecting projective varieites. 
 In this paper, we shall give the counterpart of their  Second Main Theorem  in Diophantine approximation.\\


\end{abstract}

\section{Introduction}
Let $k$ be an algebraic number field of degree $d$. Denote $M(k)$ by the set of places (i.e., equivalent classes of absolute values) of $k$ and write
$M_{\infty}(k)$ for the set of Archimedean places. From $v \in M(k)$, we choose the normalized absolute value $| . |_{v}$ such that
$| . |_{v}=| . |$ on $\Q$ (the standard absolute value) if $v$ is archimedean, whereas for $v$ non-archimedean $|p|_{v}=p^{-1}$
if $v$ lies above the rational prime $p$. Denote by $k_v$ the completion of $k$ with respect to $v$ and by $d_v=[k_v: \Q_v]$ the local
degree. We put $\Vert.\Vert_{v}=| . |_{v}^{d_v/d}$. Then norm $||.||_v$ satisfies the following properties:

$(i)$ $||x||_v\ge 0$, with equality if and only if $x=0;$

$(ii)$ $||xy||_v=||x||_v||y||_v$ for all $x, y \in k;$

$(iii)$ $||x_1+\dots+x_n||_v\le B_v^{n_v}\cdot \max \{||x_1||_v, \dots, ||x_n||_v\}$ for all $x_1, \dots, x_n \in k$, $n\in \N$, where $n_v=d_v/d$, $B_v=1$ if $v$ is non-archimedean and $B_v=n$ if $v$ is archimedean.\\
Moreover, for each $x\in k\setminus \{0\}$, we have the  following product formula:
$$ \prod_{v\in M(k)}\Vert x\Vert_{v}=1. $$
For $v\in M(k)$, we also extend $\Vert .\Vert_{v}$ to an absolute value on the algebraic closure $\overline{k}_v.$ 

For $x\in k$, the  logarithmic height of $x$ is defined by $h(x)=\sum_{v\in M(k)} \log^{+}\Vert x\Vert_{v} $, where $\log^{+}\Vert x\Vert_{v}= \log \max \{\Vert x\Vert_{v}, 1\}.$

For $x=[x_0: \dots : x_M] \in \mathbb P^n(k)$, we set $\Vert x\Vert_{v}=\max_{0 \le i \le M}\Vert x_i\Vert_{v},$ and define the logarithmic height of $x$ by
\begin{align}\label{ct11s}
h(x)=\sum_{v\in M(k)} \log \Vert x\Vert_{v}.
\end{align}

For a positive integer $d$, we set
\begin{align*}
\Cal T_d := \big\{ (i_0,\dots,i_M) \in \N_0^{M+1}:
i_0 + \dots + i_M = d \big\}.
\end{align*}
Let $Q$ be a  homogeneous polynomial of degree $d$ in $k[x_0,\dots, x_M].$
 We write
\begin{align*}
Q= \sum\limits_{I \in \Cal T_{d}} a_{I}x^I.
\end{align*}
 Set $\Vert Q\Vert_v:=\max_{I}\Vert a_I\Vert_v.$ The height of $Q$ is defined by
\begin{align*}
h(Q):=\sum_{v\in M(k)}\log\Vert Q\Vert_v.
\end{align*}
For each $v\in M(k),$ the Weil function $\lambda_{Q,v}$ is defined by
\begin{align*}
\lambda_{Q,v}(x):=\log\frac{\Vert x\Vert^d_v\cdot \Vert Q\Vert_v}{\Vert Q(x)\Vert_v}, \quad x\in \P^M(k)\;\text{such that}\;Q(x)\ne 0.
\end{align*}

Let $\Lambda$ be an infinite index set.  We call    a moving hypersurface $Q$ in $\P^M(k)$ of degree $d,$ indexed by $\Lambda$ each  collection of polynomials $\{Q(\alpha)\}_{\alpha\in \Lambda}$ in $k[x_0,\dots,x_M].$  Then, we can write $Q=\sum_{I\in\mathcal T_d}a_I x^I,$ where $a_I$'s are functions from $\Lambda$ into $k$ having no common zeros points.
 
 Through  this paper, we consider an infinite index $\Lambda;$ a set $\Cal Q:=\{Q_1,\dots,Q_q\}$ of moving hypersurfaces in $\P^M(k),$ indexed by $\Lambda;$ an arbitrary projective variety $V\subset\P^M(k)$ of dimension $n$ generated by the homogeneous  ideal $\Cal I(V)$. We write  $$Q_j= \sum_{I\in \mathcal T_{d_j}}a_{j, I}x^{I}\;(j=1,\dots,q)\;\text{where}  \;d_j=\deg Q_j.$$

Let $A\subset\Lambda$ be an infinite subset and denote by $(A,a)$ each set-theoretic map $a: A\to k.$  
Denote by $\mathcal R_{A}^{0}$  the set of equivalence classes of pairs $(C,a)$, where $C\subset A$ is a subset with finite complement and $a: C \to k$ is a map; and the equivalence relation is defined as follows: $(C_1, a_1) \sim (C_2, a_2)$
if there exists $C \subset C_1\cap C_2$ such that $C$ has finite complement in $A$ and $a_1|_{C}=a_2|_{C}.$ Then $\mathcal R_{A}^{0}$ has an obvious ring structure. Moreover, we can embed $k$ into $\mathcal R_{A}^{0}$ as constant functions.

\begin{definition} \label{defi1.2s} For each $j\in\{1,\dots,q\}$, we write $\mathcal T_{d_j}=\{I_{j,1}, \dots, I_{j, M_{d_j}}\},$ where $M_{d_j}:=\binom{d_j+M}{M}.$ A subset $A\subset \Lambda$ is said to be coherent with respect to
$\Cal Q$ if for every polynomial
$ P\in k[x_{1,1}, \dots, x_{1, M_{d_1}}, \dots, x_{q,1}, \dots, x_{q, M_{d_q}}] $
that is homogeneous in $x_{j,1}, \dots, x_{j, M_{d_j}}$ for each $j=1, \dots, q$, either\\
$ P(a_{1, I_{1,1}}(\alpha), \dots, a_{1, I_{1, M_{d_1}}}(\alpha), \dots, a_{q, I_{q,1}}(\alpha), \dots, a_{q, I_{q, M_{d_q}}}(\alpha))$
vanishes for all $\alpha \in A$ or it vanishes for only finitely many $\alpha \in A.$
\end{definition}
By \cite{RV}, Lemma 2.1 there exists an infinite coherent subset $A\subset \Lambda$ with respect to $\Cal Q.$   For each $j\in\{1,\dots,q\},$ we fix an index $I_{j}\in\Cal T_{d_j}$ such that $a_{j, I_{j}}\not\equiv  0$ (this means that
$a_{j, I_{j}}(\alpha)\ne 0$ for all, but finitely many, $\alpha\in A$), then $\dfrac{a_{j, I}}{a_{j, I_{j}}}$ defines an element of $\mathcal R_{A}^{0}$ for any
 $I \in \mathcal T_{d_j}$. This element given by the following function:
$$ \{\alpha \in A: a_{j, I_{j}}(\alpha) \ne 0\} \to k, \quad\alpha \mapsto  \frac{a_{j, I}(\alpha)}{a_{j, I_{j}}(\alpha)}.$$
Moreover, by coherent, the subring of $\mathcal R_{A}^{0}$ generated over $k$ by such all elements is an integral domain (p.3, \cite{L}). We define $\mathcal R_{A, \mathcal Q}$ to be the field of fractions of that integral domain.

Denote by $\Cal A$ the set of all functions $ \{\alpha \in A: a_{j, I_{j}}(\alpha) \ne 0\} \to k, \quad\alpha \mapsto  \frac{a_{j, I}(\alpha)}{a_{j, I_{j}}(\alpha)}$ and $k_{\Cal Q}$ the set of all formal finite sum
$\sum_{m=1}^st_m\prod_{i=1}^sc_i^{n_i},$ where $t_m\in k,$ $c_i\in \Cal A, n_i\in\N.$ 
Each pair $(\widehat b,\widehat c)\in k_{\Cal Q}^2,$  ($\hat c(\alpha)\ne 0$ for all,  but finitely many,  $\alpha\in A$) defines a set-theoretic function, denoted by $\frac{\hat b}{\hat c}$, from $\{\alpha: \widehat c(\alpha)\ne 0\}$ to $k$,   $\frac{\widehat b}{\widehat c}(\alpha):=\frac{\widehat b(\alpha)}{\widehat c(\alpha)}.$  Denote by $\widehat{\mathcal R}_{A, \mathcal Q}$ the set of all such functions. Each element $a\in\mathcal R_{A, \mathcal Q}$ is a class of some functions $\widehat{a}$ in $\widehat{\mathcal R}_{A, \mathcal Q}.$ We call that $\widehat{a}$ is a special representative of $a.$  It is clear that for any two special representatives $\widehat a_1,\widehat a_2$ of the same element $a\in\mathcal R_{A, \mathcal Q},$ we have $\widehat a_1(\alpha)=\widehat a_2(\alpha)$ for all, but finitely many $\alpha\in A.$ For a polynomial $P:=\sum_I a_Ix^I\in\mathcal R_{A,\Cal Q}[x_0, \dots, x_M],$ assume that $\widehat a_I$ is a special representative  of $a_I.$ Then $\widehat P:=\sum_I \hat a_Ix^I$ is called  a special representative of $P.$  For each $\alpha\in A$ such that all functions $\widehat a_I$'s are well defined at $\alpha,$ we set $\widehat P(\alpha):=\sum_I \hat a_I(\alpha)x^I\in k[x_0,\dots,x_M];$ and we also say that the special representative $\widehat P$  is well defined at $\alpha.$ Note that  each special representative $\widehat P$ of $P$ is well defined at all, but finitely many, $\alpha\in A.$  If  $\widehat P_1, \widehat P_2$ are two special presentatives of $P,$ then   $\widehat P_1(\alpha)= \widehat P_2(\alpha)$ for all, but finitely many $A.$
\begin{definition}\label{defi1.5s}
A sequence of points $x=[x_0:\dots : x_M]: \Lambda \to V$ is said to be $V-$algebraically non-degenerate with respect to $\mathcal Q$ if  for each infinite coherent subset $A\subset \Lambda$ with respect to $\mathcal Q$, there is no  homogeneous polynomial $P\in \mathcal R_{A,\Cal Q}[x_0, \dots, x_M] \setminus \mathcal I_{A, \Cal Q}(V)$  such that
$\widehat P(\alpha)(x_0(\alpha), \dots, x_M(\alpha)) =0$ for all, but finitely many, $\alpha\in A$ for some (then  for all) representative $\widehat P$ of $P,$ where $\mathcal I_{A, \Cal Q}(V)$ is the ideal in $\mathcal R_{A, \mathcal Q}[x_0,\dots,x_M]$ genarated by $\Cal I(V).$
\end{definition}
\begin{definition}\label{defi1.6s}
We say that the family moving hypersurfaces $\Cal Q$  is $V-$admissible if  for each $1\le j_0<\dots <j_n \le q$,  the system of equations
$$ Q_{j_i} (\alpha)(x_0, \dots, x_M)=0, \hspace{1cm} 0\le i \le n$$
has no solution $(x_0, \dots, x_M)$ satisfying $(x_0:\cdots: x_M) \in V(\overline k),$ for all, but finitely many $\alpha\in\Lambda$,  where $\overline k$ is the algebraic closure of $k.$
\end{definition}

In 1997, Ru-Vojta \cite{RV} established the following Schmidt subspace theorem for the case of moving hyperplanes in projective spaces.\\

\noindent {\bf Theorem A.}  {\it Let $k$ be a number field and let $S\subset M(k)$ be a finite set containing all archimedean places. Let $\Lambda$ be an infinite index set and let $\mathcal H:=\{H_1, \dots, H_q\}$  be a set of moving  hyperplanes  in $\mathbb P^{M}(k),$ indexed by $\Lambda.$ Let $x=[x_0: \dots: x_M]: \Lambda \to \mathbb P^{M}(k)$ be a sequence of points. Assume that

$(i)$ $x$ is linearly nondegenerate with respect to $\mathcal H,$ which mean, for each infinite coherent subset $A\subset \Lambda$ with respect to $\Cal H,$ $x_0|_A,\dots,x_M|_A$ are linearly independent over $\Cal R_{A,\Cal H},$

$(ii)$ $h(H_j(\alpha))=o(h(x(\alpha)))$ for all $\alpha\in \Lambda$ and $j=1, \dots, q$ (i.e. for all $\delta>0$, $h(H_j(\alpha))\leq\delta h(x(\alpha))$ for all, but finitely many, $\alpha\in \Lambda).$

\noindent Then, for any $\varepsilon >0,$ there exists an infinite index subset $A\subset \Lambda$ such that
$$ \sum_{v\in S}\max_K\sum_{J\in K}\lambda_{H_j(\alpha), v}(x(\alpha)) \le (M+1+\varepsilon)h(x(\alpha))$$
holds for all $\alpha \in A.$ Here the maximum is taken over all subsets $K$ of $\{1,\dots,q\}$, $\#K=M+1$ such that $H_j(\alpha), j\in K$ are linearly independent over $k$ for each $\alpha\in\Lambda.$} \\

One of the most important developments in recent years in Diophantine approximation is Schmidt's subspace theorems for fixed hypersurfaces of Corvaja-Zannier \cite{CZ} and Evertse-Ferretti \cite{EF, EF2}. Motived by their paper, Min Ru \cite{R2,R3} obtained important results  on the Second Main Theorem for fixed hypersurfaces. Later, Dethloff-Tan \cite{DT} and Cherry-Dethloff-Tan \cite{CDT} generalized these Second Main Theorems  of  Min Ru to the case of moving results. Basing on method of Dethloff-Tan \cite{DT},  Chen-Ru-Yan \cite{CRY3} and Le \cite{G}  extended Theorem A to the case of moving hypersurfaces in the projective spaces. In this paper,  following the method of Cherry-Dethloff- Tan \cite{CDT}, we give the following counterpart of their  Second Main Theorem  in Diophantine approximation.  

\begin{theorem}\label{Schmidt} Let $k$ be a number field and let $S \subset M(k)$ be a finite set containing all archimedean places. Let $x=[x_0: \dots: x_M]: \Lambda \to V$ be a sequence
of points. Assume that

$(i)$ $\Cal Q$ is $V-$admissible,  and $x$ is $V-$ algebraically nondegenerate with respect to $\mathcal Q$;

$(ii)$ $h(Q_j(\alpha))=o(h(x(\alpha)))$ for all $\alpha\in\Lambda$ and $j=1, \dots,q$ (i.e. for all $\delta>0$, $h(Q_j(\alpha))\leq\delta h(x(\alpha))$ for all, but finitely many, $\alpha\in \Lambda).$

\noindent Then, for any $\varepsilon >0,$ there exists an infinite index subset $A\subset \Lambda$ such that
$$ \sum_{v\in S}\sum_{j=1}^{q}\lambda_{Q_j(\alpha), v}(Q_j(x(\alpha)) \le (n+1+\varepsilon)h(x(\alpha))$$
holds for all $\alpha \in A.$
\end{theorem}
\begin{remark}\label{remark}  (i) By replacing $Q_j$ by $Q_j^{\frac{d}{d_j}},$ where $d=\text{lcm}\{d_1,\dots,d_q\},$ in Theorem \ref{Schmidt}, we may assume that $Q_1,\dots,Q_q$ have the same degree $d.$

(ii) By replacing $Q_j= \sum_{I\in \mathcal T_{d}}a_{j, I}x^{I}$ by $Q'_j = \sum\limits_{I \in \Cal T_{d}} \dfrac{a_{jI}}{a_{jI_{j}}}x^I$, in Theorem \ref{Schmidt}, we may assume that $Q_j\in\mathcal R_{A, \mathcal Q}[x_0,\dots,x_M].$
\end{remark}

\noindent {\bf Acknowledgements:} This research was supported by Vietnam National Foundation for Science and Technology Development (NAFOSTED) under grant number 101.02-2016.17. The second and the third named authors were partially supported by the Vietnam Institute for Advanced Studies in Mathematics.
Tran Van Tan  is currently Regular Associate Member of ICTP, Trieste, Italy. We would like to thank Gerd Dethloff for helpful comments on the first version of our paper.

\section{Some Lemmas}
We write
\begin{align*}
Q_j = \sum\limits_{I \in \Cal T_{d}} a_{jI}x^I, \quad (j = 0,\dots,q).
\end{align*}
Let $A\subset\Lambda$ be an infinity coherent subset with respect to $\Cal Q.$  For each $j\in\{1,\dots,q\},$ we fix an index $I_{j}\in\Cal T_{d}$ such that $a_{j, I_{j}}\not\equiv  0$ (this means that
$a_{j, I_{j}}(\alpha)\ne 0$ for all but finitely many $\alpha\in A$), then $\dfrac{a_{jI}}{a_{jI_{j}}}$ defines an element of $\mathcal R_{A}^{0}$ for any
 $I \in \mathcal T_{d}$.
 Set \begin{align*}
Q'_j = \sum\limits_{I \in \Cal T_{d}} \dfrac{a_{jI}}{a_{jI_{j}}}x^I, \quad (j = 0,\dots,q).
\end{align*}
\noindent Let $ t=(\dots, t_{jI},\dots)$ be a family of variables.
 Set
\begin{align*}
\widetilde{Q_j} = \sum\limits_{I \in \Cal T_{d}} t_{jI}x^I\in k[t,x].
\end{align*}
We have
$\widetilde{Q_j}(\dots,\dfrac{a_{jI}}{a_{jI_{j}}}(\alpha),\dots,x_0,\dots, x_M)=Q'_j(\alpha)(x_0,\dots,x_M)$ for all $\alpha\in A$ outside a finite subset.

Assume that the ideal $\Cal I(V)$ of $V$ is generated by homogeneous polynomials $P_1,\dots,P_m.$ Since $\Cal Q$ is  $V-$admissible, for each $J:=\{j_0,\dots,j_n\}\subset\{1,\dots,q\}$ there is a subset  $A_J\subset A$ with finite complement such that for all $\alpha\in A_J$ 
  the homogeneous polynomials 
 $P_1,\dots,P_m,Q'_{j_0}(\alpha),\dots,Q'_{j_n}(\alpha)\in k[x_0,\dots,x_M]$
 have no common non-trivial solutions in $\overline{k}^{M+1}.$  
Denote by $_{k[t]}(P_1,\dots, P_m, \widetilde{Q}_{j_0},\dots,\widetilde{Q}_{j_n})$
 the ideal in the ring of  polynomials in $x_0,\dots, x_M$ with coefficients in $k[t]$ generated by
$P_1,\dots,P_m,\widetilde{Q}_{j_0},\dots,\widetilde{Q}_{j_n}.$
 A polynomial $\widetilde R$ in $k[t]$  is called an
 {\it inertia form} of the polynomials $P_1,\dots, P_m, \widetilde{Q}_{j_0},\dots,\widetilde{Q}_{j_n}$ if it has the following property:
\begin{align*}
x_i^s\cdot \widetilde {R}\in \;_{k[t]}(P_1,\dots, P_m, \widetilde {Q}_{j_0},\dots,\widetilde{Q}_{j_n})
\end{align*}
for $i=0,\dots,M$ and for some non-negative integer $s$  (see e.g.  \cite{Z}).\\
It follows from the definition that the set $\Cal I$ of inertia forms of polynomials $P_1,\dots, P_m, \widetilde{Q}_{j_0},\dots,\widetilde{Q}_{j_n}$ is an ideal in $k[t].$

It is well known that $(m+n+1)$ homogeneous polynomials $P_i(x_0,\dots,x_M),$ $\widetilde{Q_j}(\dots,t_{jI},\dots,x_0,\dots, x_M),$ $ i\in\{1,\dots,m\},$
$j\in J$
have no
common non-trivial solutions in $x_0,\dots,x_M$
for special values $t_{jI}^0$ of $t_{jI}$
if and only if there exists an inertia form $\widetilde {R_J}^{t_{jI}^0}$  such that
$\widetilde R_J^{t_{jI}^0}(\dots,t_{jI}^0,\dots)\ne 0$ (see e.g. \cite{Z}, page 254). 
For  each $\alpha\in A_J,$ choose  $\widetilde {R_J}^{\alpha}\in\Cal I$ with respect to the special values $t_{jI}^\alpha:= \dfrac{a_{jI}}{a_{jI_{j}}}(\alpha).$ 
Set $R_J^{\alpha}:=\widetilde R_J^{\alpha}(\dots,\dfrac{a_{jI}}{a_{jI_{j}}},\dots).$ Then $R_J^{\alpha}$  is a special presentative of an element  $\mathcal R_{A,\Cal Q}.$ 
By construction, we have
\begin{align}\label{t1}
R_J^{\alpha}(\alpha)\ne 0 \;\text{ for all} \;\alpha\in A_J \;(\text{and hence, for all, but finitely many,}\;\alpha\in A) .
\end{align}
Since $k[t]$ is Noetherian,  $\Cal I$ is generated by finite polynomials $\widetilde{R}_{J_1},\dots,\widetilde{R}_{J_s}.$ For each $\alpha,$ we write
$\widetilde {R_J}^{\alpha}=\sum_{\ell=1}^s \widetilde{G}^{\alpha}_\ell\widetilde {R}_{J_\ell},$  $\widetilde G^{\alpha}_{\ell}\in k[t].$ We have that $G^\alpha_{\ell}:=\widetilde G^{\alpha}_\ell(\dots,\dfrac{a_{jI}}{a_{jI_{j}}},\dots),$ and $R_{J\ell}:=\widetilde {R}_{J\ell}(\dots,\dfrac{a_{jI}}{a_{jI_{j}}},\dots)$ are special representatives of elements in $\mathcal R_{A,\Cal Q}.$ It is clear that $R_J^{\alpha}=\sum_{\ell=1}^s G^{\alpha}_{\ell} R_{J\ell}.$
Hence, by (\ref{t1}), we have
\begin{align*}
0\ne R_J^{\alpha}(\alpha)=\sum_{\ell=1}^sG^\alpha_{\ell} (\alpha) R_{J_\ell}(\alpha)
\end{align*}
for all, but finitely many $\alpha\in A.$ 
Therefore, there is $\ell_0\in\{1,\dots,s\}$ such that  
\begin{align}\label{t2}
R_{J_{\ell_0}}(\alpha)\ne 0 \quad \text{for all, but finitely many}\quad \alpha\in A.
\end{align}
 Furthermore, by the definition of the inertia forms, there are a non-negative integer $s,$ polynomials
 $b_{i\ell}\in\mathcal R_{A, \Cal Q}[x_0, \dots, x_M]$ with $\deg b_{ij_k}=s-d$ and $\deg b_{i\ell}=s-\deg P_\ell$ such that
\begin{align}\label{ct17s}
R_{J_{\ell_0}}\cdot x_i^{s}=\sum_{k=0}^{n}b_{i j_k}\cdot Q'_{j_k}+\sum_{\ell=1}^{m}b_{i\ell}\cdot P_\ell,
\end{align}
for all $0\le i\le M$.

Let $x: \Lambda \to V\subset P^{M}(k)$ be a map. A map $(C, a) \in \mathcal R_{A}^{0}$ is called small with respect to $x$ if and only if
$$ h(a(\alpha))=o(h(x(\alpha))), $$
which mean that, for every $\varepsilon >0$, there exists a subset $C_{\varepsilon}\subset C$ with finite complement such that $h(a(\alpha))\le \varepsilon h(x(\alpha))$ for all $\alpha \in C_{\varepsilon}.$ We denote by $\mathcal K_x$ the set of all such small maps. Then, $\mathcal K_x$ is subring of $ \mathcal R_{A}^{0}$. It is not an entire ring, however, if $(C, a)\in \mathcal K_x$ and $a(\alpha)\ne 0$ for all but finitely $\alpha \in C$, then we have $(C\setminus \{\alpha: a(\alpha)=0\}, \dfrac{1}{a})\in \mathcal K_x.$
Denote by $\mathcal C_x$ the set of all positive functions $g$ defined over $\Lambda$ outside a finite subset of $\Lambda$ such that
$$ \log^{+}(g(\alpha)) =o(h(x(\alpha))).$$
Then $\mathcal C_x$ is a ring. Moreover, if $(C, a)\in \mathcal K_x$, then for every $v\in M(k)$, the function $||a||_v: C\to \mathbb R^{+}$ given by $\alpha \mapsto ||a(\alpha)||_v$ belongs to $\mathcal C_x$. Furthermore, if $(C, a)\in \mathcal K_x$ and $a(\alpha)\ne 0$ for all but finitely $\alpha \in C$, the function $g: \{\alpha|a(\alpha) \ne 0\}\mapsto \dfrac{1}{||a(\alpha)||_v}$ also belongs to $\mathcal C_x.$

From (\ref{t2}) and (\ref{ct17s}), similarly to Lemma 2.2 in \cite{G}, under the assumption of Theorem \ref{Schmidt} and Remark \ref{remark}, we have the following result.
\begin{lemma}\label{lem21s} Let $A\subset \Lambda$ be coherent with respect to $\Cal Q.$ Then for each $J\subset\{1,\dots,q\}$, there are functions $l_{1, v}, l_{2, v}\in\mathcal C_x$ such that
$$ l_{2, v}(\alpha)||x(\alpha)||_v^{d}\le \max_{j\in J}||Q_j(\alpha)(x(\alpha))||_v\le l_{1, v}(\alpha)||x(\alpha)||_v^{d},$$
for all $\alpha \in A$ ouside finite subset and all $v\in S.$
\end{lemma}

For each positive integer $\ell$ and for each vector sub-space $W$ in $k[x_0, \dots, x_M]$ (or in $\mathcal R_{A, \Cal Q}[x_0, \dots, x_M]$), we denote by $W_\ell$ the vector space consisting of all holomogeneous polynomials in $W$ of degree $\ell$ (and  of the zero polynomial).

By the usual theory of Hilbert polynomials, for $N>>0$, we have
\begin{align*} H_V(N)&:=\dim_k \dfrac{k[x_0, \dots, x_M]_N}{\mathcal I(V)_N}\\
&=\dim_{\overline {k}} \dfrac{\overline {k}[x_0, \dots, x_M]_N}{\mathcal I(V(\overline k))_N}
=\deg V. \dfrac{N^n}{n!}+O(N^{n-1}).
\end{align*}

\begin{definition}\label{defi21s}
Let $W$ be a vector sub-space in $\mathcal R_{A, \Cal Q}[x_0,\dots,x_M].$  For each $\alpha \in A,$  we denote
$$W(\alpha):=\cup_{P\in W}\{\widehat P(\alpha):   \widehat P\; \text {is a special representative of } P,\;\text{well defined at}\; \alpha\}.$$
\end{definition}
 It is clear that $W(\alpha)$ is a vector sub-space of $k[x_0,\dots,x_M].$
\begin{lemma} \label{L0s}
Let $W$ be a vector sub-space in $\mathcal R_{A, \Cal Q}[x_0,\dots,x_M]_N.$   Then two following assertions hold:

$(i)$  There are $\gamma_j\in\mathcal R_{A, \Cal Q}[x_0,\dots,x_M]_N$, $j=1,\dots,H$ such that  $\gamma_j(\alpha),\dots, \gamma_H(\alpha)$ form a basis of $W(\alpha)$, for all, but finitely many, $\alpha\in A$ (i.e.
for any representative  $\widehat\gamma_j$ of
 $\gamma_j$, then  $\widehat\gamma_j(\alpha),\dots, \widehat\gamma_H(\alpha)$ form a basis of $W(\alpha)$, for all, but finitely many, $\alpha\in A)$. In particular, dimension of $W(\alpha)$ does not depend on $\alpha\in A$ outside a finite subset.

$(ii)$ Let $\{h_j\}_{j=1}^K$ be a basis of $W.$  Then $\{h_j(\alpha)\}_{j=1}^K$ forms a basis of $W(\alpha)$
(hence, $\dim_{\mathcal R_{A, \Cal Q}} W=\dim_{k}W(\alpha)$) for all, but finitely many, $\alpha\in A$   (i.e. for any special representative $\widehat h_j$ of   $h_j$, then $\{\widehat h_j(\alpha)\}_{j=1}^K$ forms a basis of $W(\alpha)$
 for all, but finitely many, $\alpha\in A$).  
\end{lemma}
\begin{proof} Set $H:=\max_{\alpha\in A}\dim W(\alpha).$ Taking $\alpha_0\in A$ such that $\dim W(\alpha_0)=H.$ Then, there are $\gamma_j\in\mathcal R_{A, \Cal Q}[x_0,\dots,x_M]_N$ $(j=1,\dots,H)$ and there are special representatives $\widehat\gamma_j$ of $\gamma_j$ $(j=1,\dots,H)$ such that $\{\widehat\gamma_1(\alpha_0),\dots, \widehat\gamma_H(\alpha_0)\}$ form a basis of $W(\alpha_0).$ 

Denote by $B$ the matrix of coefficients of $\{\widehat\gamma_j\}_{j=1}^H.$ Then, $B(\alpha_0)$ has rank $H.$ Hence, there is a square sub-matrix $B_1$ of $B$ with order $H$ such that $\det B_1(\alpha_0)\ne 0.$ By coherent of $A,$ there is a complement $A_1$  of a finite subset in $A,$ such that
$\det B(\alpha)\ne 0$ and all coefficients of $\widehat\gamma_j$'s are well defined at $\alpha,$ for all $\alpha\in A_1,$  Then, $\{\widehat\gamma_1(\alpha),\dots, \widehat\gamma_H(\alpha)\}$ are linearly independent, for all $\alpha\in A_1.$ On the other hand  $\dim W(\alpha)\leq H.$  Hence, $\{\widehat\gamma_1(\alpha),\dots, \widehat\gamma_H(\alpha)\}$ is a basis of $W(\alpha)$  for all $\alpha\in A_1.$ On the other hand, for any special representative $\widehat\gamma_j'$ of $\gamma_j,$ then $\widehat\gamma_j'(\alpha)=\widehat\gamma_j(\alpha)$ for all, but finitely many, $\alpha\in A$. Hence,  $\widehat\gamma_1'(\alpha),\dots, \widehat\gamma_H'(\alpha)$ also form a basis of $W(\alpha)$ for all, but finitely many, $\alpha\in A.$  This completes the proof of   assertion (i).

 Let $(c_{ij})$ be the matrix of coefficients of $\{h_j\}_{j=1}^K.$ Since
$\{h_j\}_{j=1}^K$ are linearly independent, there exists a square submatrix $C$ of $(c_{ij})$ with order $K$ and $\det C\not\equiv 0.$  Let $\widehat c_{ij}$ be an special representative  of $c_{ij}.$ Denote by $\widehat C$ the matrix which is defined from $C$ by replacing $c_{ij}$ by $\widehat c_{ij}.$ Then $\det\widehat C$ is a special representative of $\det C,$ and hence $\det\widehat C\not\equiv 0.$ By coherent of A,   $\det \hat C(\alpha)\ne 0$ and all coefficients of $\widehat h_{j}'s$ are well defined at $\alpha,$ for all, but finitely many, $\alpha\in A.$ By assertion (i), there are $\gamma_j\in\mathcal R_{A, \Cal Q}[x_0,\dots,x_M]_N$, $j=1,\dots,H$ such that  $\widehat\gamma_1(\alpha),\dots,\widehat\gamma_H(\alpha)$ form a basis of $W(\alpha)$ for all, but finitely many, $\alpha\in A$, where $\widehat\gamma_j$ is a special representative of $\gamma_j$.  We  write $\gamma_s=\sum_{j=1}^Kt_{sj}h_j$ with $t_{sj}\in \mathcal R_{A, \Cal Q}.$ Let $\widehat t_{sj}$ be a special representative of $t_{sj}$ $(s\in\{1,\dots,H\}, j\in\{1,\dots,K\}).$ Then $\sum_{j=1}^K\widehat t_{sj}\widehat h_j$ is a special representative of $\gamma_s.$ Hence, $\widehat\gamma_s(\alpha)=\sum_{j=1}^K\widehat t_{sj}(\alpha)\widehat h_j(\alpha)$ $(s=1,\dots,H)$ for all, but finitely many, $\alpha\in A$.   Combining with (i), we have that $\{\widehat h_j(\alpha)\}_{j=1}^K$ is a generating system of $W(\alpha)$ for all, but finitely many, $\alpha\in A$. On the other hand, since $\det \widehat C(\alpha)\ne 0$  and  all coefficients of $\widehat h_{ij}'s$ are well defined at $\alpha,$ for all, but finitely many, $\alpha\in A.$ Hence,  $\widehat h_1(\alpha),\dots,\widehat h_K(\alpha)$ are linearly independent for all, but finitely many, $\alpha\in A.$ By these facts, $\{\widehat h_1(\alpha),\dots,\widehat h_K(\alpha)\}$ is a basis of $W(\alpha),$ for all, but finitely many, $\alpha\in A.$
\end{proof}

Denote by $\mathcal I_{A, \Cal Q}(V)$  the ideal  in $\mathcal R_{A, \Cal Q }[x_0, \dots, x_M]$ generated by the elements in $\mathcal I(V).$ It is clear that $\mathcal I_{A, \Cal Q}(V)$
is also the sub-vector space of $\mathcal R_{A, \Cal Q}[x_0, \dots, x_M]$ generated by $\mathcal I (V ).$

We use the lexicographic order in $\N_0^n$ and for $I=(i_1,\dots,i_n),$ set $\Vert I\Vert :=i_1+\cdots+i_n.$
\begin{definition} For each 
 $I=(i_1,\cdots, i_n)\in \N_0^n$ and $N\in\N_0$ with $N\geq d\Vert I\Vert,$ denote by $\Cal L_N^I$ 
the set of all $\gamma\in\mathcal R_{A, \Cal Q }[x_0,\dots,x_M]_{N-d\Vert I\Vert}$ such that 
\begin{align*}
Q_1^{i_1}\cdots Q_n^{i_n}\gamma-
\sum_{E=(e_1,\dots,e_n)>I}Q_1^{e_1}\cdots Q_n^{e_n}\gamma_E\in \mathcal I_{A, \Cal Q }(V)_N.
\end{align*}
for some $\gamma_E\in \mathcal R_{A, \Cal Q} [x_0,\dots,x_M]_{N-d\Vert E\Vert}$.
\end{definition}
Denote by $\Cal L^I$ the homogeneous ideal in $\mathcal R_{A, \Cal Q } [x_0,\dots,x_M]$ generated by $\cup_{N\geq d\Vert I\Vert}\Cal L_N^I.$ 
\begin{remark}\label{r1} i) $\Cal L_N^I$ is a $\mathcal R_{A, \Cal Q }$-vector sub-space of $\mathcal R_{A, \Cal Q }[x_0,\dots,x_M]_{N-d\Vert I\Vert},$ and
 $(\Cal I(V), Q_1,\dots,Q_n)_{N-d\Vert I\Vert}\subset \Cal L_N^I,$   where $(\Cal I(V), Q_1,\dots,Q_n)$ is
 the ideal in $\mathcal R_{A, \Cal Q }[x_0,\dots,x_M]$ generated by
 $\Cal I(V)\cup\{Q_1,\dots,Q_n\}.$ 

ii) For any $\gamma\in \Cal L_N^I$ and $P\in \mathcal R_{A, \Cal Q }[x_0,\dots,x_M]_k,$ we have $\gamma\cdot P\in\Cal L_{N+k}^I$

iii) $\Cal L^I\cap \mathcal R_{A, \Cal Q }[x_0,\dots,x_M]_{N-d\Vert I\Vert}=\Cal L_N^I.$
   
iv) $\frac{\mathcal R_{A, \Cal Q }[x_0, \dots, x_M]}{\Cal L^I}$ 
is a graded modul over the graded ring $\mathcal R_{A, \Cal Q }[x_0,\dots,x_M].$
\end{remark}
\begin{lemma} \label{inftylemma} $\#\{\Cal L^I: I\in \N_0^n\}<\infty.$ 
\end{lemma}
\begin{proof} 
Suppose that $\#\{\Cal L^I: I\in \N_0^n\}=\infty.$ Then there exists an infinite sequence $\{\Cal L^{I_k}\}_{k=1}^{\infty}$ consisting of pairwise different ideals. We write
$I_k=(i_{k1},\dots,i_{kn}).$ Since $i_{k\ell}\in\N_0$, there exists an infinite sequence of positve integers $p_1<p_2<p_3<\cdots$ such that 
$i_{p_1\ell}\leq i_{p_2\ell}\leq i_{p_3\ell}\leq\cdots$, for all $\ell=1,\dots,n$\,: In fact, firstly we choose  a sub-sequence $i_{q_11}\leq i_{q_21}\leq i_{q_31}\leq \cdots$  of $\{i_{k1}\}_{k=1}^{\infty}.$ Next, we choose a sub-sequence of $i_{r_12}\leq i_{r_22}\leq i_{r_32}\leq \cdots$ of $\{i_{q_k2}\}_{k=1}^{\infty}.$  Continuing the above process until obtaining a sub-sequence $i_{p_1n}\leq i_{p_2n}\leq i_{p_3n}\leq\cdots.$

We now prove that: 
\begin{align}\label{Noertherian}
\Cal L^{I_{p_1}}\subset \Cal L^{I_{p_2}}\subset\Cal L^{I_{p_3}}\subset\cdots.
\end{align}
Indeed, for any  $\gamma\in\Cal L_N^{I_{p_k}}$ (for any $N$ and $k$ satisfying $N-\Vert I_{p_k}\Vert \geq 0$), we have
\begin{align*}
Q_1^{i_{p_k1}}\cdots Q_n^{i_{p_kn}}\gamma-
\sum_{E=(e_1,\dots,e_n)>I_{p_k}}Q_1^{e_1}\cdots Q_n^{e_n}\gamma_E\in \mathcal I_{\mathcal R_{A, \Cal Q}}(V)_N,
\end{align*}
 for some $\gamma_E\in \mathcal R_{A, \Cal Q } [x_0,\dots,x_M]_{N-d\Vert E\Vert}.$ 

\noindent Then, since $i_{p_{k+1}1}-i_{p_k1},\dots, i_{p_{k+1}n}-i_{p_kn}$ are non-negative integers, we have
\begin{align*}
Q_1^{i_{p_{k+1}1}}\cdots Q_n^{i_{p_{k+1}n}}\gamma-
\sum_{E=(e_1,\dots,e_n)>I_{p_k}}Q_1^{e_1+(i_{p_{k+1}1}-i_{p_k1})}\cdots Q_n^{e_n+(i_{p_{k+1}n}-i_{p_kn})}\gamma_E\in \mathcal I_{\mathcal R_{A, \Cal Q}}(V)_N.
\end{align*}
On the other hand since $E=(e_1,\dots,e_n)>I_{p_k}$ we have $(e_1+i_{p_{k+1}1}-i_{p_k1},\dots,e_n+i_{p_{k+1}n}-i_{p_kn})>I_{p_{k+1}}$. Therefore, $\gamma\in\Cal L_{N-d\Vert I_{p_k}\Vert+d\Vert I_{p_{k+1}}\Vert}^{I_{p_{k+1}}}.$
Hence, $\Cal L_N^{I_{p_k}}\subset \Cal L_{N-d\Vert I_{p_k}\Vert+d\Vert I_{p_{k+1}}\Vert}^{I_{p_{k+1}}}$ for all $k, N.$ Therefore, $\Cal L^{I_{p_k}}\subset \Cal L^{I_{p_{k+1}}}$ for all $k$. We get (\ref{Noertherian}). 

\noindent Since $\mathcal R_{A, \Cal Q }[x_0,\dots,x_M]$ is a noetherian ring, the chain of ideals in (\ref{Noertherian}) becomes finally stationary. This is a contradiction.
\end{proof}

Set \begin{align*}
m_N^I:=\dim_{\mathcal R_{A, \Cal Q}}\frac{\mathcal R_{A, \Cal Q }[x_0,\dots,x_M]_{N-d\Vert I\Vert}}{\Cal L_N^I}.
\end{align*}
For each positive integer $N,$ denote by $\tau_N$ the set of all $I:=(i_0,\dots,i_n)\in \N_0^n$ with $N-d\Vert I\Vert\geq 0.$

Let $\gamma_{I1},\dots,\gamma_{Im_N^I}\in  \mathcal R_{A, \Cal Q}[x_0,\dots,x_M]_{N-d\Vert I\Vert}$ such that they form a basis of the $\mathcal R_{A, \Cal Q}-$ vector space $\frac{ \mathcal R_{A, \Cal Q}[x_0,\dots,x_M]_{N-d\Vert I\Vert}}{{\Cal L_N^I}}.$

Similarly to \cite{CDT}, Lemma 2.6, we have:
\begin{lemma}\label{L1s}
$\{[Q_{1}^{i_1}\cdots Q_{n}^{i_n}\cdot\gamma_{I1}],\dots,[Q_1^{i_1}\cdots Q_n^{i_n}\cdot\gamma_{Im_N^I}],\;I=(i_1,\dots,i_n)\in\tau_N\}$
is a basis of the $\mathcal R_{A, \Cal Q}$- vector space
$\frac{\mathcal R_{A, \Cal Q}[x_0,\dots,x_M]_N}{\mathcal I_{A, \Cal Q}(V)_N}.$
\end{lemma}
\begin{proof}
Firstly, we prove that:
\begin{align}\label{a0s}
\{[Q_{1}^{i_1}\cdots Q_{n}^{i_n}\cdot\gamma_{I1}],\dots,[Q_{1}^{i_1}\cdots Q_{n}^{i_n}\cdot\gamma_{Im_N^I}],\;I=(i_1,\dots,i_n)\in\tau_N\},
\end{align}
 are linealy independent.

\noindent Indeed,
for any $t_{I\ell}\in  \mathcal R_{A, \Cal Q},$ $(I=(i_1,\dots,i_n)\in\tau_N, \ell\in\{1,\dots,m_N^I\})$ such that
\begin{align*}
\sum_{I\in\tau}\big(t_{I1}[Q_{1}^{i_1}\cdots Q_{n}^{i_n}\cdot\gamma_{I1}]+\cdots+t_{Im_N^I}[Q_{1}^{i_1}\cdots Q_{n}^{i_n}\cdot\gamma_{Im_N^I}]\big)=0.
\end{align*}
Then
\begin{align}\label{a1s}
\sum_{I\in\tau_N}Q_{1}^{i_1}\cdots Q_{n}^{i_n}\big(t_{I1}\gamma_{I1}+
\cdots+ t_{Im_N^I}\gamma_{Im_N^I}\big)\in\mathcal I_{A, \Cal Q}(V)_N.
\end{align}
By the definition of $\Cal L_N^I,$ and by (\ref{a1s}), we get
\begin{align*}
t_{I^*1}\gamma_{I^*1}+
\cdots+ t_{I^*m_N^{I^*}}\gamma_{I^*m_N^{I^*}}\in\mathcal L_N^{I^*},
\end{align*}
where $I^*$ is the smallest elements of $\tau_N.$\\
On the other hand, $\{\gamma_{I^*1},\dots,\gamma_{I^*m_N^{I^*}}\}$ makes a basis of $\frac{  \mathcal R_{A, \Cal Q}[x_0,\dots,x_M]_{N-d\Vert I^*\Vert}}{\mathcal L_N^{I^*}}.$
 \\
Hence,
\begin{align}\label{a2s}
t_{I^*1}=\cdots=t_{I^*m_N^{I^*}}=0.
\end{align}
Then, by (\ref{a1s}), we have
\begin{align*}
\sum_{I\in\tau_N\setminus \{I^*\}}Q_{1}^{i_1}\cdots Q_{n}^{i_n}\big(t_{I1}\gamma_{I1}+
\cdots+ t_{Im_N^I}\gamma_{Im_N^I}\big)\in\mathcal I_{A, \Cal Q}(V)_N.
\end{align*}
Then, similarly to (\ref{a2s}), we have
\begin{align*}
t_{\tilde I1}=\cdots=t_{\tilde Im_N^{\tilde I}}=0,
\end{align*}
where $\tilde I$ is the smallest element of $\tau_N\setminus\{I^*\}.$

\noindent Continueting the above process, we get that $t_{I\ell}=0$ for all $I\in\tau_N$ and $\ell\in\{1,\dots, m_N^I\},$ and hence, we get (\ref{a0s}).

Denote by $\Cal L$ the vector sub-space in $ \mathcal R_{A, \Cal Q}[x_0,\dots,x_M]_N$ generated by
$$\{Q_{1}^{i_1}\cdots Q_{n}^{i_n}\cdot\gamma_{I1},\dots,Q_{1}^{i_1}\cdots Q_{n}^{i_n}\cdot\gamma_{Im_N^I},\;I=(i_1,\dots,i_n)\in\tau_N\}.$$

Now we prove that: For any $I=(i_1,\dots,i_n)\in\tau_N,$  we have
\begin{align}\label{a3s}
Q_{1}^{i_1}\cdots Q_{n}^{i_n}\cdot\gamma_I\in \Cal L+\mathcal I_{A, \Cal Q}(V)_N,
\end{align}
for all $\gamma_I\in \mathcal R_{A, \Cal Q}[x_0,\dots,x_M]_{N-d\Vert I\Vert}.$

Set $I'=(i'_1,\dots,i'_n):=\max\{I: I\in\tau_N\}.$
Since, $\gamma_{I'1},\dots,\gamma_{I'm_N^{I'}}$
form a basis of $\frac{\mathcal R_{A, \Cal Q}[x_0,\dots,x_M]_{N-d\Vert I\Vert}}{{\Cal L_N^{I'}}},$  for any  $\gamma_{I'}\in \mathcal R_{A,\Cal Q}[x_0,\dots,x_M]_{N-d\Vert I'\Vert},$ we have
\begin{align}\label{a4s}
\gamma_{I'}=\sum_{\ell=1}^{m_N^{I'}}t_{I'\ell}\cdot\gamma_{I'\ell}+h_{I'\ell},\; \text{where} \;h_{I'\ell}\in\Cal L_N^{I'},\;\text{and}\;t_{I'\ell}
\in \mathcal R_{A, \Cal Q}.
\end{align}
On the other hand, by the defintion of $\Cal L_N^{I'}$ , we have
 $Q_{1}^{i'_1}\cdots Q_{n}^{i'_n}\cdot h_{I'\ell}\in \mathcal I_{A, \Cal Q}(V)_N$
(note that $I'=\max\{I: I\in\tau_N\}).$
\noindent Hence,
\begin{align*}
Q_{1}^{i'_1}\cdots Q_{n}^{i'_n}\cdot\gamma_{I'}=\sum_{\ell=1}^{m_N^{I'}}t_{I'\ell}Q_1^{i'_1}\cdots Q_n^{i'_n}
\cdot\gamma_{I'\ell}+Q_{1}^{i'_1}\cdots Q_{n}^{i'_n}\cdot h_{I'\ell}
\in  \Cal L+\mathcal I_{A, \Cal Q}(V)_N.
\end{align*}
We get (\ref{a3s}) for the case where $I=I'.$

Assume that (\ref{a3s}) holds for all $I>I^*.$ We prove that (\ref{a3s}) holds also for $I=I^*=(i^*_1,\dots,i^*_n).$

\noindent Indeed, similarly to (\ref{a4s}), for any  $\gamma_{I^*}\in \mathcal R_{A, \Cal Q}[x_0,\dots,x_M]_{N-d\Vert I^*\Vert},$ we have
\begin{align*}
\gamma_{I^*}=\sum_{\ell=1}^{m_N^{I^*}}t_{I^*\ell}\cdot\gamma_{I^*\ell}+h_{I^*\ell},\; \text{where} \;h_{I^*\ell}\in \Cal L_N^{I^*},
\;\text{and}\;t_{I^*\ell}\in \mathcal R_{A, \Cal Q} .
\end{align*}
Then,
\begin{align}\label{a5s}
Q_{1}^{i^*_1}\cdots Q_{n}^{i^*_n}\cdot\gamma_{I^*}=\sum_{\ell=1}^{m_N^{I^*}}t_{I_s\ell}Q_{1}^{i^*_1}\cdots Q_{n}^{i^*_n}
\cdot\gamma_{I^*\ell}+Q_{1}^{i^*_1}\cdots Q_{n}^{i^*_n}\cdot h_{I^*\ell}.
\end{align}
Since $h_{I^*\ell}\in\Cal L_N^{I^*},$ we have
\begin{align*}
Q_{1}^{i^*_1}\cdots Q_{n}^{i^*_n}\cdot h_{I^*\ell}=\sum_{E=(e_1,\dots,e_n)>I^*}Q_{1}^{e_1}\cdots Q_{n}^{e_n}\cdot g_{E},
\end{align*}
for some $g_E\in \mathcal R_{A, \Cal Q}[x_0,\dots,x_M]_{N-d\cdot\Vert E\Vert}.$

\noindent Therefore,  by the induction hypothesis,
\begin{align*}
Q_{1}^{i^*_1}\cdots Q_{n}^{i^*_n}\cdot h_{I^*\ell}\in \mathcal L+\mathcal I_{A, \Cal Q}(V)_N.
\end{align*}
Then, by (\ref{a5s}), we have
\begin{align*}
Q_{1}^{i^*_1}\cdots Q_{n}^{i^*_n}\cdot\gamma_{I^*}\in\mathcal L+\mathcal I_{A, \Cal Q}(V)_N.
\end{align*}
This means that (\ref{a3s}) holds for $I=I^*.$ Hence, by  induction we get (\ref{a3s}).

For any $Q\in \mathcal R_{A, \Cal Q}[x_0,\dots,x_M]_N,$ we write $Q=Q_{1}^0\cdots Q_{n}^0\cdot Q.$
Then by (\ref{a3s}), we have
$$Q\in\mathcal L+\mathcal I_{ A, \Cal Q}(V)_N.$$
Hence,
\begin{align*}
\{[Q_{1}^{i_1}\cdots Q_{n}^{i_n}\cdot\gamma_{I1}],\dots,[Q_{1}^{i_1}\cdots Q_{n}^{i_n}\cdot\gamma_{Im_N^I}],\;I=(i_1,\dots,i_n)\in\tau_N\}
\end{align*}
is a generating system of $\frac{\mathcal R_{A, \Cal Q}[x_0,\dots,x_M]_N}{\mathcal I_{A, \Cal Q}(V)_N}.$
Combining with (\ref{a1s}), we get the conclution of Lemma \ref{L1s}.
\end{proof}
Similarly to \cite{CDT}, Lemma, 2.8 we have:
\begin{lemma}\label{newlm} There are   integers $n_0$, $c$  and $c'$ such that the following assertions hold.

i) \;$\dim_{\mathcal R_{A, \Cal Q}[x_0,\dots,x_M]} \frac{\mathcal R_{A, \Cal Q}[x_0,\dots,x_M][x_0,\dots,x_M]_{N-d\Vert I\Vert}}{(\Cal I(V), Q_1,\dots,Q_n)_{N-d\Vert I\Vert}}=c$ 
for all $I\in\N_0^n, N\in\N_0$ satisfying $N-d\Vert I\Vert\geq n_0.$

ii) \;For each $I\in \N_0^n$ there is an integer $m^I$ such that $m^I=m_N^I$ for all $N\in\N_0$ satisfying $N-d\Vert I\Vert\geq n_0.$

iii) \; $m_N^I\leq c',$ for all $I\in\N_0^n$ and $N\in\N_0$ satisfying $N-d\cdot \Vert I\Vert\geq 0.$
\end{lemma}
\begin{proof} Denote by $(\Cal I(V), Q_{1},\dots, Q_{n})$ the ideal in $\mathcal R_{A, \Cal Q}[x_0,\dots,x_M]$ generated by $\mathcal I (V)\cup\{ Q_{1},\dots, Q_{n}\}.$ 
 For each $\alpha$ in $A$ such that all coefficients of $Q_{i}$'s are well defined at $\alpha,$ we denote  by
$(\mathcal I (V), Q_{1}(\alpha),\dots, Q_{n}(\alpha))$
the ideal in $k[x_0,\dots,x_M]$ generated by $ \mathcal I(V)\cup\{Q_{1}(\alpha),\dots,Q_{n}(\alpha)\}.$

We have
\begin{align}\label{a7s}
(\mathcal I (V),Q_{1}(\alpha),\dots,Q_{n}(\alpha))\subset ( \mathcal I (V),Q_{1},\dots,Q_{n})(\alpha).
\end{align}
Indeed,  for any $P\in(\mathcal I (V),Q_{1}(\alpha),\dots,Q_{n}(\alpha)),$ we write $P=G+Q_{1}(\alpha)\cdot P_1+\cdots+Q_{n}(\alpha)\cdot P_n,$ where $G\in \mathcal I (V),$
and $P_i\in k[x_0,\dots,x_M].$ Take $P'\in (I(V), Q_{1},\dots, Q_{n})$ which is defined by a presentation $\widehat P':=G+Q_{1}\cdot P_1+\cdots+Q_{n}\cdot P_n.$  It is clear that $\widehat P'(\alpha)=P.$
Hence, we get (\ref{a7s}).

\noindent Let $I$ be an arbitrary element in $\tau_N.$ Let $\{h_k:=\sum_{i=1}^n Q_{i}\cdot R_{ik}+\sum_{j=1}^{m_k}  g_{jk}\cdot\gamma_{jk}\}_{k=1}^K$
 be a basis of $(\mathcal I (V), Q_{1},\dots,Q_{n})_{N-d\cdot \Vert I\Vert},$ where $g_{jk}\in  \Cal I(V),$ and
 $R_{ik}, \gamma_{jk},\in  \mathcal R_{A, \Cal Q}[x_0,\dots,x_M]$ satisfying $\deg (Q_{i}\cdot R_{ik})=\deg (\gamma_{jk}\cdot g_{jk})=N-d\cdot \Vert I\Vert.$ Let
 $\widehat R_{ik},$ and $\widehat\gamma_{jk}$ be some special representatives of $ R_{ik},$ and $\gamma_{jk}$, respectively. Then
$\widehat h_k:=\sum_{i=1}^n Q_{i}\cdot \widehat R_{ik}+\sum_{j=1}^{m_k}  g_{jk}\cdot\widehat \gamma_{jk}$ is a representative of $h_k.$
By Lemma \ref{L0s}, and since $\Cal Q$ is a $V-$ admissible set, there exists $\alpha \in A$ such that:

i) $\{\widehat h_k(\alpha)\}_{k=1}^K$  is a basis  of $( \mathcal I (V),Q_1,\dots, Q_n)_{N-d\cdot \Vert I\Vert}(\alpha),$

ii)  all coefficients of
$Q_j, \hat R_{jk}, \widehat\gamma_{jk},  g_{jk}$ are well defined at $\alpha,$ and

iii) homogeneous polynomials $Q_{0}(\alpha),\dots,Q_{n}(\alpha)\in k[x_0,\dots,x_M]$ have no common zeros point in $V(\overline{k}).$ 

 \noindent On the other hand, it is clear that $\hat h_k(\alpha)\in(\mathcal I (V),Q_{1}(\alpha),\dots,Q_{n}(\alpha)),$ for all $k=1,\dots,K.$
Hence, by (\ref{a7s}), and by i), we have
\begin{align*}
(\mathcal I (V), Q_{1}(\alpha),\dots,Q_{n}(\alpha))_{N-d\cdot \Vert I\Vert}= (\mathcal I (V),Q_{1},\dots,Q_{n})_{N-d\cdot \Vert I\Vert}(\alpha).
\end{align*}
Then, we have
\begin{align*}
 \dim_{ \mathcal R_{A,\Cal Q }} (\mathcal I(V),Q_{1},\dots,Q_{n})_{N-d\cdot\Vert I\Vert}
&=\dim_{k}(\mathcal I (V),Q_{1},\dots,Q_{n})_{N-d\cdot\Vert I\Vert}(\alpha)\\
&=\dim_{k}(\mathcal I (V),Q_{1}(\alpha),\dots,Q_{n}(\alpha))_{N-d\cdot\Vert I\Vert}.
\end{align*}
This implies that
\begin{align}\label{aa8s}
\dim\frac{ \mathcal R_{A, \Cal Q} [x_0,\dots,x_M]_{N-d\Vert I\Vert}}{(\mathcal I(V),Q_{1},\dots,Q_{n})_{N-d\cdot\Vert I\Vert}}
&= \dim\frac{k[x_0,\dots,x_M]_{N-d\Vert I\Vert}}{(\mathcal I (V),Q_{1}(\alpha),\dots,Q_{n}(\alpha))_N}\notag\\
&=\dim\frac{\overline{k}[ x_0,\dots,x_M]_{N-d\Vert I\Vert}}{(\mathcal I (V(\overline{k})),Q_{1}(\alpha),\dots,Q_{n}(\alpha))_{N-d\cdot\Vert I\Vert}}.
\end{align}
On the other hand, by the Hilbert-Serre Theorem (\cite{H}, Theorem 7.5), there exist positive integers $n_1, c$  such that
\begin{align*}
\dim\frac{\overline{k}[ x_0,\dots,x_M]_{N-d\Vert I\Vert}}{(\mathcal I (V(\overline{k})),Q_{1}(\alpha),\dots,Q_{n}(\alpha))_{N-d\cdot\Vert I\Vert}}=c,
\end{align*}
for all $I\in\N_0^n$ and $ N\in\N_0$ satisfying $N-d\Vert I\Vert\geq n_1.$

\noindent Combining with (\ref{aa8s}), we have
\begin{align}\label{aa8}
\dim\frac{ \mathcal R_{A, \Cal Q}[ x_0,\dots,x_M]_{N-d\Vert I\Vert}}{(\mathcal I (V),Q_{1},\dots,Q_{n})_{N-d\cdot\Vert I\Vert}}
=c,
\end{align}
for all $I\in\N_0^n$ and $N\in\N_0$  satisfying $N-d\Vert I\Vert\geq n_1.$

Let $h^I$ and $h$ be the Hilbert functions of $\frac{ \mathcal R_{A, \Cal Q}[x_0,\dots,x_M]}{\Cal L^I}$ and $\frac{ \mathcal R_{A, \Cal Q}[x_0,\dots,x_M]}{(\Cal I(V), Q_1,\dots,Q_n)},$ respectively. Since  $(\Cal I(V), Q_1,\dots,Q_n)\subset \Cal L^I,$ we have $h^I\leq h.$ On the other hand, by Matsumura \cite{Ma}, Theorem 14,  $h^I(k)$ is a polynomial in $k$ for all $k>>0$ and by (2.12), we have $h(k)=c$ for all $k\geq n_1.$ Hence, there are constants $m^I$, $n_2$ such that $h^I(k)=m^I$ for all $k\geq n_2$ and then $m_N^I=h^I(N-d\Vert I\Vert)=m^I$  for all $N\in \N_0$ satisfying $N-d\Vert I\Vert\geq n_2.$ By Lemma~\ref{inftylemma}, we may choose $n_2$ common for all $I.$ 
Taking $n_0:=\max\{n_1,n_2\},$ we get Lemma \ref{newlm}, i) and ii).

We have $m^I_N=h^I(N-d\Vert I\Vert)\leq h(N-d\Vert I\Vert)\leq\max\{c, h(k): k=0,\dots, n_0\}.$ Hence, taking $c':=\max\{c, h(k): k=0,\dots, n_0\},$ we get  
Lemma~\ref{newlm}, iii).
\end{proof}
Set
\begin{align*} m:=\min_{I \in \N_0^n}m^I.
\end{align*}
We fix  $I_0=(i_{01},\dots,i_{0n})\in \N_0^n,$ and  $N_0\in \N_0$ such that $N_0-d\Vert I_0\Vert\geq n_0$ and $m^{I_0}_{N_0}=m.$

For each  positive integer $N,$  divisible by $d,$ denote by $\tau_N^0$ the set of all $I=(i_1,\dots,i_n)\in\tau_N$ such that $N-d\Vert I\Vert \geq n_0$ and $i_k\geq \max\{i_{01},\dots,i_{0n}\}$,  for all $k\in\{1,\dots, n\}.$ 

\noindent We have 
\begin{align}\label{ph1} \#\tau_N=\binom{\frac{N}{d}+n}{n}&=\frac{1}{d^n}\cdot\frac{N^n}{n!}+O(N^{n-1}), \notag\\
                                        \#\{I\in\tau_N: N-d\Vert I\Vert\leq n_0\}&=O(N^{n-1}),\notag\\
                                         \#\{I=(i_1,\dots,i_n)\in\tau_N:\; i_k&<\max_{1\leq\ell\leq n}i_{0\ell}, \;\text{for some}\;k\}=O(N^{n-1}),\;\text{and\: so}\notag\\
                                         \#\tau_N^0&=\frac{1}{d^n}\cdot\frac{N^n}{n!}+O(N^{n-1}).
\end{align}
Similarly to \cite{CDT}, Lemma 2.9, we have:
\begin{lemma} \label{L2s}
  $m_N^I=\deg V\cdot d^n$ for all $N>>0,$  divisible by $d,$ and  $I\in \tau^0_N.$
\end{lemma}
\begin{proof}
 For any $\gamma\in\Cal L^{I^0}_{N_0},$ we have
\begin{align*}
T:=Q_1^{i_{01}}\cdots Q_n^{i_{0n}}\gamma-\sum\limits_{E=(e_{1},\dots,e_{n})>I_0}Q_1^{e_{1}}\cdots Q_n^{e_{n}}\gamma_E\in\Cal I_{A, \Cal Q}(V)_{N_0},
\end{align*}
for some $\gamma_E\in\mathcal R_{A, \Cal Q}[x_0,\dots,x_M]_{N-d\Vert E\Vert}.$

\noindent Then, for any $I=(i_1,\dots,i_n)\in\tau_N^0,$ we have
\begin{align}\label{h3}
Q_1^{i_{1}}\cdots Q_n^{i_{n}}\gamma&-\sum\limits_{E=(e_{1},\dots,e_{n})>I_0}Q_1^{e_{1}+i_1-i_{01}}\cdots Q_n^{e_{n}+i_n-i_{0n}}\gamma_E\notag\\
&=Q_1^{i_{1}-i_{01}}\cdots Q_n^{i_{n}-i_{0n}}\cdot T\in\Cal I_{A, \Cal Q}(V)_{N_0}.
\end{align}
On the other hand since $I\in\tau_N^0$ and $E>I_0,$ we have $(e_{1}+i_1-i_{01},\dots,e_{n}+i_n-i_{0n})>I.$

\noindent Hence, by (\ref{h3}) we have
$$\gamma\in\Cal L^I_{N_0+d\Vert I\Vert-d\Vert I_0\Vert}.$$
This implies that 
\begin{align*}
\Cal L^{I_0}_{N_0}\subset \Cal L^I_{N_0+d\Vert I\Vert-d\Vert I_0\Vert}.
\end{align*}
Then
\begin{align}
m=m^{I_0}_{N_0}&=\dim_{\mathcal R_{A, \Cal Q}}\frac{\mathcal R_{A, \Cal Q}[x_0,\dots,x_M]_{N_0-d\Vert I_0\Vert}}{\Cal L^{I_0}_{N_0}}\notag \\
&\geq\dim_{\mathcal R_{A, \Cal Q}}\frac{\mathcal R_{A, \Cal Q}[x_0,\dots,x_M]_{N_0-d\Vert I_0\Vert}}{\Cal L^I_{N_0+d\Vert I\Vert-d\Vert I_0\Vert}}\notag \label{tt}\\
&=m^I_{N_0+d\Vert I\Vert-d\Vert I_0\Vert}.
\end{align}
On the other hand since $\left(N_0+d\Vert I\Vert-d\Vert I_0\Vert\right)-d\Vert I\Vert =N_0-d\Vert I_0\Vert\geq n_0,$ and  $N-\Vert I\Vert\geq n_0$  (note that $I\in\tau_N^0$), by Lemma \ref{newlm}, we have
 $$m^I=m^I_{N_0+d\Vert I\Vert-d\Vert I_0\Vert}=m^I_{N}.$$
Hence, by (\ref{tt}), $m\geq m^I=m^I_{N}.$ Then, by the minimum property of $m,$ we get that 
\begin{align}\label{np1}
m^I_N=m\; \text{for all}\; I\in\tau_N^0.
\end{align}

We now prove that:  
\begin{align}\label{cl} \dim_{\mathcal R_{A, \Cal Q}}\Cal I_{A, \Cal Q}(V)_N=\dim_{k}\Cal I(V)_N. 
\end{align} 
Indeed,  let $\{P_1,\dots, P_s\}$ be a basis of the $k$ vector space $\Cal I(V)_N.$ 
It is clear that $\Cal I_{\mathcal R_{A, \Cal Q}}(V)_N$ is a vector space over $\mathcal R_{A, \Cal Q}$ generated by $\Cal I(V)_N,$ therefore $\{P_1,\dots, P_s\}$ is 
also a generating system of $\Cal I_{A, \Cal Q}(V)_N.$ Then,  for (\ref{cl}), it suffices to prove that
if $t_1,\dots, t_s\in\mathcal R_{A, \Cal Q}$ satisfy
\begin{align}\label{a6}
t_1\cdot P_1+\cdots+ t_s\cdot P_s\equiv 0,
\end{align} 
then $t_1=\dots=t_s\equiv 0.$
We rewrite (\ref{a6}) in the following form
\begin{align*}
C\cdot \left ( \begin{matrix}
t_1\cr
\cdot\cr
\cdot\cr
\cdot\cr
t_s\end{matrix}\right)=\left ( \begin{matrix}
0\cr
\cdot\cr
\cdot\cr
\cdot\cr
0\end{matrix}\right),
\end{align*}
where $C\in$Mat$(\binom{M+N}{N}\times s,\mathcal R_{A, \Cal Q}).$

\noindent If the above system of linear equations has non-trivial solutions, then rank$_{\mathcal R_{A, \Cal Q}} C<s.$ Then 
rank$_{k} C(z)<s$ for all, but finitely many $z\in A$.  Taking $a\in A$ such that 
rank$_{k} C(a)<s.$ Then the following system of linear equations 
\begin{align*}
C(a)\cdot \left ( \begin{matrix}
t_1\cr
\cdot\cr
\cdot\cr
\cdot\cr
t_s\end{matrix}\right)=\left ( \begin{matrix}
0\cr
\cdot\cr
\cdot\cr
\cdot\cr
0\end{matrix}\right),
\end{align*}
has some non-trivial solution $(t_1,\dots,t_s)=(\alpha_1,\dots,\alpha_s)\in k^s\setminus\{0\}.$ 
Then $\alpha_1\cdot P_1+\cdots+\alpha_s\cdot P_s\equiv 0,$ this is
a contradiction. Hence, we get (\ref{cl}).

By Lemma \ref{L1s} and (\ref{cl}), we have
\begin{align}\label{ttt}
\sum_{I\in\tau_N}m_N^I=\dim_{\mathcal R_{A, \Cal Q}}\frac{\mathcal R_{A, \Cal Q}[x_0,\dots,x_M]_N}{\Cal I_{A, \Cal Q}(V)_N}
&=\dim_{k}\frac{k[x_0,\dots,x_M]_N}{\Cal I(V)_N}\notag\\
&=\dim_{\overline{k}}\frac{\overline{k}[x_0,\dots,x_M]_N}{\Cal I(V(\overline{k}))_N}\notag\\
&=\deg V\cdot \frac{N^n}{n!} +O(N^{n-1}), 
\end{align}
for all $N$ large enough. \\
Combining with (\ref{np1}), we have
\begin{align}\label{aa9}
m\cdot\#\tau_N^0+\sum_{I\in\tau_N\setminus\tau_N^0}m_N^I=\deg V\cdot \frac{N^n}{n!} +O(N^{n-1}).
\end{align}
On the other hand by Lemma \ref{newlm}, $m_N^I\leq c',$ for all $I\in\tau_N\setminus\tau_N^0.$ Hence, by (\ref{ph1}), we have
\begin{align*}
m=\deg V\cdot d^n.
\end{align*}
Combining with (\ref{np1}), we have
\begin{align*}
m^I_N=\deg V\cdot d^n
\end{align*}
for all $I\in\tau_N^0.$
\end{proof}

Similarly to \cite{CDT}, Lemma 2.10, we have
\begin{lemma}\label{L3s}
For each $s\in\{1,\dots,n\},$ and for $N>>0,$ divisible by $d,$ we have:
\begin{align*}
\sum_{I=(i_1,\dots,i_n)\in\tau_N}m_N^I\cdot i_s\geq \frac{\deg V}{d\cdot (n+1)!}N^{n+1}-O(N^n).
\end{align*}
\end{lemma}
\begin{proof}
Firstly, we note that if $I=(i_1,\dots,i_n)\in\tau^0_N,$ then all symmetry $I'=(i_{\sigma(1)},\dots,i_{\sigma(n)})$ of $I$ also belongs to $\tau^0_N.$ On the other hand, by 
 Lemma~\ref{L2s}, we have $m_N^I=\deg V\cdot d^n,$ for all $I\in\tau_N^0.$ Therefore, by (\ref{ph1}) we have 
\begin{align*}
\sum_{I=(i_1,\dots,i_n)\in\tau_N^0}&m_N^I\cdot i_1=\cdots=\sum_{I=(i_1,\dots,i_n)\in\tau_N^0}m_N^I\cdot i_n\notag\\
&=\deg V\cdot d^n\cdot \sum_{I\in\tau_N^0}\frac{\Vert I\Vert}{n}\notag\\
&=\deg V\cdot d^n\cdot \left(\sum_{I\in\tau_N}\frac{\Vert I\Vert}{n}-
 \sum_{I\in\tau_N\setminus\tau^0_N}\frac{\Vert I\Vert}{n}\right)\notag\\
&\geq\deg V\cdot d^n\left(\sum_{k=0}^{\frac{N}{d}}\frac{k}{n}
\cdot\binom{k+n-1}{n-1}-\left(\#\tau_N-\#\tau_N^0\right)
\cdot \frac{N}{nd}\right)\notag\\
&=\deg V\cdot d^n\left(\sum_{k=0}^{\frac{N}{d}}\frac{k}{n}
\cdot\binom{k+n-1}{n-1}-O(N^{n-1})
\cdot \frac{N}{nd}\right)\notag\\
&=\deg V\cdot d^n\sum_{k=1}^{\frac{N}{d}}\binom{k+n-1}{n}-O(N^{n})\notag\\
&=\deg V \cdot d^n \binom{\frac{N}{d}+n}{n+1}-O(N^{n})\notag\\
&\geq \frac{\deg V}{d\cdot (n+1)!}N^{n+1}-O(N^n).
\end{align*}
Hence, for each $i\in\{1,\dots,n\}$
\begin{align*}
\sum_{I=(i_1,\dots,i_n)\in\tau_N}m_N^I\cdot i_s&\geq \sum_{I=(i_1,\dots,i_n)\in\tau_N^0}m_N^I\cdot i_s\notag
\\
&\geq \frac{\deg V}{d\cdot (n+1)!}N^{n+1}-O(N^n).
\end{align*}
\end{proof}
 We recall that by (\ref{ttt}), for $N>>0,$ we have
$$\dim_{\mathcal R_{A, \Cal Q}}\frac{\mathcal R_{A, \Cal Q}[x_0,\dots,x_M]_N}{\mathcal I_{A, \Cal Q}(V)_N}
=H_V(N)
=\deg V\cdot \frac{N^n}{n!} +O(N^{n-1}).$$
Therefore, from Lemmas \ref{L1s}, \ref{L3s} we get immediately the following result.
\begin{lemma}\label{L4s}
For all $N>>0$ divisible $d,$ there are homogeneous polynomials $\phi_1,\dots,\phi_{H_V(N)}$  in $ \mathcal R_{A, \Cal Q}[x_0,\dots,x_M]_N$ such that they form a basis
 of the $\mathcal R_{A, \Cal Q}-$ vector space
$\frac{\mathcal R_{A, \Cal Q}[x_0,\dots,x_M]_N}{\mathcal I_{A, \Cal Q}(V)_N},$  and
 \begin{align*}
\prod_{j=1}^{H_V(N)}\phi_j-\big(Q_{1}\cdots Q_{n}\big)^{\frac{\deg V\cdot N^{n+1}}{d\cdot (n+1)!}-u(N)}\cdot P\in\mathcal I_{A, \Cal Q}(V)_N,
\end{align*}
where  $u(N)$ is a function  satisfying $u(N)\leq O(N^n)$,  $P \in \mathcal R_{A, \Cal Q}[x_0,\dots,x_M]$ is a homogeneous polynomials of degree
$$N\cdot H_V(N)-\frac{n\cdot\deg V\cdot N^{n+1}}{(n+1)!}+u(N)=\frac{\deg V\cdot N^{n+1}}{(n+1)!}+O(N^n).$$
\end{lemma}

\section{Proof of our main theorem}
\begin{proof}
By Lemma 2.1 in \cite{CRY3}, there exists an infinite index subset $A\subset \Lambda$ which is coherent with respect to $\mathcal Q.$   By Remark \ref{remark}, we may assume that the polynomials $Q_j$'s have the same degree $d\ge 1$ and their coefficients  belong to the field  $\mathcal R_{A, \Cal Q}.$  By the fact that for any infinite subset $B$ of $A$, then $B$ is also coherent with respect to $\mathcal Q$ and $\Cal R_{B,\Cal Q}\subset\Cal R_{A,\Cal Q},$ in our proof, we may freely pass to infinite subsets. For simplicity, we still denote these infinite subsets by $A.$

From the assumption, for each $a\in \mathcal R_{A, \Cal Q},$ and $v\in M(k)$, we have, for all $\alpha \in A$,
\begin{align}\label{ct21s}
\log ||a(\alpha)||_v\le \sum_{v\in M(k)} \log^{+}||a(\alpha)||_v=h(a(\alpha))\le o(h(x(\alpha))).
\end{align}
For each $v\in S$, and $\alpha \in A$, there exist a subset $J(v, \alpha)=\{j_1(v, \alpha), \dots, j_n(v, \alpha)\} \subset \{1,\dots,q\}$ such that
\begin{align*}
 0< ||Q_{j_1(v, \alpha)}(\alpha)(x(\alpha))||_v\le \dots& \le ||Q_{j_n(v, \alpha)}(\alpha)(x(\alpha))||_v\\
&\le \min_{j\not \in \{j_1(v, \alpha), \dots, j_n(v, \alpha)\}} ||Q_{j}(\alpha)(x(\alpha))||_v.
\end{align*}
By Lemma \ref{lem21s}, we have
\begin{align}\label{ct22s}
\log \prod_{j=1}^{q}||Q_j(\alpha)(x(\alpha))||_v&=\log \prod_{\beta_j\not \in J(v, \alpha)}||Q_{\beta_j}(\alpha)(x(\alpha))||_v+\log \prod_{i=1}^{n}||Q_{j_i(v, \alpha)}(\alpha)(x(\alpha))||_v\notag\\
&\ge (q-n)d \log ||x(\alpha)||_v-\log \overset{\sim}h_v(x(\alpha))\notag\\
&+\log \prod_{i=1}^{n}||Q_{j_i(v, \alpha)}(\alpha)(x(\alpha))||_v,
\end{align}
where $ \overset{\sim}h_v=\prod(1+h_{\mu})$, $h_{\mu}$ runs over all the choices of $l_{2, v}$, thus $ \overset{\sim}h_v \in \mathcal C_x.$

\noindent By Lemma \ref{L4s}, there exist homogeneous polynomials $\phi_1^{J(v, \alpha)}, \dots, \phi_{H_V(N)}^{J(v, \alpha)}$(depend on $J(v, \alpha)$) in $\mathcal R_{A, \Cal Q}[x_0, \dots, x_M]_N$ and there are functions $u(N), v(N)$ (common for all $J(v,\alpha))$ such that $\{\phi_i^{J(v, \alpha)}\}$  a basic of $\mathcal R_{A, \Cal Q}-$ vector space $\dfrac{\mathcal R_{A, \Cal Q}[x_0, \dots, x_M]_N}{\mathcal I_{A, \Cal Q}(V)_N}$ and
$$ \prod_{\ell=1}^{H_V(N)}\phi_\ell^{J(v, \alpha)}-(Q_{j_1(v,\alpha)}\dots Q_{j_n(v, \alpha)})^{\dfrac{\deg V \cdot N^{n+1}}{d(n+1)!}-u(N)}P_{J(v,\alpha)} \in  \mathcal I_{A, \Cal Q}(V)_N,$$
where  $P_{J(v,\alpha)}\in \mathcal R_{A, \Cal Q}[x_0,\dots,x_M]$ is a homogeneous polynomials of degree
$\frac{\deg V\cdot N^{n+1}}{(n+1)!}+v(N).$

\noindent Thus, for all $x(\alpha) \in V(k)$, we have
\begin{align*} \prod_{\ell=1}^{H_V(N)}\phi_\ell^{J(v, \alpha)}(\alpha)(x(\alpha))=\Big(\prod_{i=1}^nQ_{j_i(v,\alpha)}(\alpha)(x(\alpha))\Big)^{\dfrac{\deg V \cdot N^{n+1}}{d(n+1)!}-u(N)}P_{J(v,\alpha)}(\alpha)(x(\alpha)) .
\end{align*}
On the other hand, it is easy to that there exist $h_{J(v, \alpha)}\in \mathcal C_x$ such that
\begin{align*}
||P_{J(v,\alpha)}(\alpha)(x(\alpha))||_v &\leq ||(x(\alpha))||_v^{\deg P_{J(v,\alpha)}}h_{J(v, \alpha)}(\alpha)\\
&=||(x(\alpha))||_v^{\frac{\deg V\cdot N^{n+1}}{(n+1)!}+v(N)}h_{J(v, \alpha)}(\alpha).
\end{align*}
Therefore,
\begin{align*}
\log \prod_{\ell=1}^{H_N(V)}||\phi_\ell^{J(v, \alpha)}(\alpha)(x(\alpha))||_v
&\leq\Big( \dfrac{\deg V . N^{n+1}}{d(n+1)!}-u(N)\Big)\cdot\log ||\prod_{i=1}^nQ_{j_i(v,\alpha)}(\alpha)(x(\alpha))||_v\\
&+\log^+ h_{J(v, \alpha)}(\alpha)+( \dfrac{\deg V . N^{n+1}}{(n+1)!}+v(N))\log ||x(\alpha))||_v
\end{align*}
This implies that there are functions $\omega_1(N),\omega_2(N)\leq O(\frac{1}{N})$ such that
\begin{align}\label{ct23s}
\log ||\prod_{i=1}^nQ_{j_i(v,\alpha)}(\alpha)(x(\alpha))||_v&\geq  \Big(\dfrac{d(n+1)!}{\deg V . N^{n+1}}-\frac{\omega_1(N)}{N^{n+1}}\Big)\cdot\log \prod_{\ell=1}^{H_N(V)}||\phi_\ell^{J(v, \alpha)}(\alpha)(x(\alpha))||_v \notag\\
&-\log ^+\overset{\sim}h_{J(v, \alpha)}(\alpha)-(d+\omega_2(N))\log ||x(\alpha))||_v,
\end{align}
for some $\overset{\sim}h_{J(v, \alpha)} \in \mathcal C_x$.

\noindent By (\ref{ct22s}) and (\ref{ct23s}), we have
\begin{align}\label{ct26s}
\log \prod_{j=1}^{q}||Q_j(\alpha)(x(\alpha))||_v&\ge (q-n-1)d \log ||x(\alpha)||_v-\log^+ \overset{\sim}h_v(x(\alpha))\notag\\
&+\Big(\dfrac{d(n+1)!}{\deg V . N^{n+1}}-\frac{\omega_1(N)}{N^{n+1}}\Big)\cdot\log \prod_{\ell=1}^{H_N(V)}||\phi_\ell^{J(v, \alpha)}(\alpha)(x(\alpha))||_v \notag\\
&-\log^+ \overset{\sim}h_{J(v, \alpha)}(\alpha)-\omega_2(N)\log ||x(\alpha))||_v.
\end{align}

We fix homogeneous polynomials $\Phi_1, \dots, \Phi_{H_N(V)} \in \mathcal R_{A, \{Q_j\}_{j=1}^{q}}[x_0, \dots, x_M]_N$ such that they form a basic of $\mathcal R_{A, \Cal Q}-$ vector space $\dfrac{\mathcal R_{A, \Cal Q}[x_0, \dots, x_M]_N}{{{\mathcal I}_{A, \Cal Q}(V)}_N}.$ Then, there exist homogeneous linear polynomials
$$L_1^{J(v, \alpha)}, \dots, L_{H_N(V)}^{J(v, \alpha)}\in \mathcal R_{A, \Cal Q}[y_1, \dots, y_{H_V(N)}]$$ such that they are linear independent over $\mathcal R_{A, \{Q_j\}_{j=1}^{q}}$ and
$$ \phi_\ell^{J(v, \alpha)}-L_\ell^{J(v,\alpha)}(\Phi_1, \dots, \Phi_{H_N(V)}) \in \mathcal I_{A, \Cal Q}(V)_N,$$
for all $\ell=1, \dots, H_V(N).$ It is clear that  $h(L_\ell^{J(v, \alpha)}(\beta))=o(h(x(\beta))), \beta\in A$ and $A$ is coherent with respect to $\{L_\ell\}_{\ell=1}^{H_N(V)}.$

We have,
\begin{align}\label{ct24s}
 \prod_{\ell=1}^{H_V(N)}||\phi_\ell^{J(v, \alpha)}(\alpha)(x(\alpha))||_v=\prod_{\ell=1}^{H_N(V)}||L_\ell^{J(v, \alpha)}(\Phi_1, \dots, \Phi_{H_V(N)})(\alpha)(x(\alpha))||_v.
\end{align}
We  write
$$L_\ell^{J(v, \alpha)}(y_1, \dots, y_{H_V(N)})=\sum_{s=1}^{H_V(N)}g_{\ell r}y_s, \quad g_{\ell s}\in\mathcal R_{A, \Cal Q}.$$
Since $L_1^{J(v, \alpha)}, \dots, L_{H_N(V)}^{J(v, \alpha)}$ are linear independent over $\mathcal R_{A, \Cal Q},$   we have $\det (h_{\ell s})\ne 0\in \mathcal R_{A,\Cal Q}.$
 Thus, due to cohenrent property  of $A,$
$\det(h_{\ell s})(\beta)\ne 0$  for all $\beta\in A$, outside a finite subset of $A.$ By passing to an infinite subset if necessary, we may assume that $L_1^{J(v,\alpha)}(\beta),\dots,L_{H_N(V)}^{J(v,\alpha)}(\beta)$ are lineraly independent over $k$ for all $\beta\in A.$

Now we consider the sequence of points
$F(\alpha)=[\Phi_1(x(\alpha)), \dots, \Phi_{H_N(V)}(x(\alpha))]$ from $A$ to $\P^{H_V(N)-1}(k)$ and
moving hyperplanes $\Cal L:=\{L_1^{J(v, \alpha)}, \dots, L_{H_N(V)}^{J(v, \alpha)}\}$ in $\P^{H_V(N)-1}(k),$ indexed by $A.$ We claim that $F$ is linearly nondegenerate with respect to $\Cal L.$ Indeed, ortherwise, then there is a linear form $L\in\Cal R_{B, \Cal L}[y_1,\dots, y_{H_N(V)}]$ for some infinite coherent subset $B\subset A,$ such that $L(F)|_B\equiv 0$ in $B,$ which contradicts to the assumption that $x$ is algebraically nondegenerate with respect to $\Cal Q.$

By Theorem A, for any $\epsilon >0,$ there is an infinite subset of $A$ (common for all $J(v,\alpha)),$ denoted again by $A$, such that
\begin{align}\label{ct25s}
\sum_{v\in S} \log \prod_{\ell=1}^{H_V(N)}\dfrac{||F(\alpha)||_v||L_\ell^{J(v, \alpha)}(\alpha)||_v}{||L_\ell^{J(v, \alpha)}(\alpha)(F(\alpha))||_v}\le (H_V(N)+\varepsilon)h(F(\alpha)),
\end{align}
for all $\alpha \in A.$

\noindent Combining with (\ref{ct26s}) and (\ref{ct24s}) we have
\begin{align}\label{ct27s}
\sum_{v\in S}\sum_{j=1}^{q}\log \dfrac{||x(\alpha)||_v^{d}}{||Q_j(\alpha)(x(\alpha))||_v}&\le (n+1)d \sum_{v\in S}\log ||x(\alpha)||_v+\omega_2(N)\sum_{v\in S}\log ||x(\alpha))||_v\notag\\
&+\Big(\dfrac{d(n+1)!}{\deg V \cdot N^{n+1}}-\frac{\omega_1(N)}{N^{n+1}}\Big)\sum_{v\in S}\log \prod_{\ell=1}^{H_N(V)}\dfrac{||F(\alpha)||_v||L_\ell^{J(v,\alpha)}(\alpha)||_v}{||L_\ell^{J(v, \alpha)}(\alpha)(x(\alpha))||_v}\notag\\
&-H_V(N)\Big(\dfrac{d(n+1)!}{\deg V \cdot N^{n+1}}-\frac{\omega_1(N)}{N^{n+1}}\Big)\sum_{v\in S}\log ||F(\alpha)||_v\notag\\
&+o(h(x(\alpha))).
\end{align}
Since the above inequality is independent of the choice of components of $x(\alpha)$, we can choose the components of $x(\alpha)$ being $S-$ integers so that
\begin{align}\label{phuong}\sum_{v\in S}\log ||x(\alpha)||_v&=h(x(\alpha)),  \text{and}
\sum_{v\in S}\log\Vert F(\alpha)\Vert_v&= h(F(\alpha)) \le Nh(x(\alpha))+o(h(x(\alpha))).
\end{align}
Combining with (\ref{ct27s}) and (\ref{ct25s}), we obtain
\begin{align*}
\sum_{v\in S}\sum_{j=1}^{q}\log \dfrac{||x(\alpha)||_v^{d}}{||Q_j(\alpha)(x(\alpha))||_v}&\le (n+1)dh(x(\alpha))+\omega_2(N) h(x(\alpha))\\
&+\Big(\dfrac{d(n+1)!}{\deg V \cdot N^{n+1}}-\frac{\omega_1(N)}{N^{n+1}}\Big)(H_V(N)+\varepsilon)h(F(\alpha))\\
&-H_V(N)\Big(\dfrac{d(n+1)!}{\deg V \cdot N^{n+1}}-\frac{\omega_1(N)}{N^{n+1}}\Big)h(F(\alpha))\\
&+o(h(x(\alpha)).
\end{align*}
Hence, we obtain
\begin{align}\label{ct28s}
\sum_{v\in S}\sum_{j=1}^{q}\log \dfrac{||x(\alpha)||_v^{d}||Q_j(\alpha)||_v}{||Q_j(\alpha)(x(\alpha))||_v}&\le  (n+1)dh(x(\alpha))+\omega_2(N)h(x(\alpha))\notag\\
&+\epsilon\Big(\dfrac{d(n+1)!}{\deg V . N^{n+1}}-\frac{\omega_1(N)}{N^{n+1}}\Big)h(F(\alpha))+o(h(x(\alpha))).
\end{align}
Here, we note that $h(Q_j(\alpha))=o(h(x(\alpha))).$

\noindent Combining with (\ref{phuong}), by our choice with $\omega_1,\omega_2$ for  $N$ large enough, we get
\begin{align*}
\sum_{v\in S}\sum_{j=1}^{q}\log \dfrac{||x(\alpha)||_v^{d}||Q_j(\alpha)||_v}{||Q_j(\alpha)(x(\alpha))||_v}\le  (n+1+\varepsilon)d h(x(\alpha)),
\end{align*}
for all $\alpha\in A.$
This completes the proof of Theorem \ref{Schmidt}.
\end{proof}

\noindent Nguyen Thanh Son\\
Department of Mathematics\\
  Hanoi National University of Education\\
 136-Xuan Thuy street, Cau Giay, Hanoi, Vietnam\\
e-mail: k16toannguyenthanhson@gmail.com\\

\noindent Tran Van Tan\\
Department of Mathematics\\
  Hanoi National University of Education\\
 136-Xuan Thuy street, Cau Giay, Hanoi, Vietnam\\
e-mail: tranvantanhn@yahoo.com\\

\noindent Nguyen Van Thin\\
Department of Mathematics\\
 Thai Nguyen University of Education\\
Luong Ngoc Quyen Street, 
Thai Nguyen city, Vietnam.\\
e-mail: thinmath@gmail.com\\

\end{document}